\documentclass[a4paper,11pt]{amsart}
\usepackage{amsmath,amsthm,amsfonts,amssymb}

\theoremstyle{plain}

\newtheorem{thm}{Theorem}
\newtheorem{lem}[thm]{Lemma}

\newtheorem{cor}[thm]{Corollary}

\newcommand{\N}{\mathbb N}

\title{Omega subgroups of powerful $p$-groups}
\author{Gustavo A. Fern\'andez-Alcober}
\address{Matematika Saila\\ Euskal Herriko Unibertsitatea\\ 48080 Bilbao\\ Spain}
\email{gustavo.fernandez@ehu.es}
\thanks{Supported by the Spanish Ministry of Science and Education, grant MTM2004-04665,
partly with FEDER funds, and by the University of the Basque Country, grant UPV05/99.}

\begin{document}

\begin{abstract}
Let $G$ be a powerful finite $p$-group.
In this note, we give a short elementary proof of the following facts for all $i\ge 0$:
(i) $\exp \Omega_i(G)\le p^i$ for odd $p$, and $\exp \Omega_i(G)\le 2^{i+1}$
for $p=2$; (ii) the index $|G:G^{p^i}|$ coincides with the number of
elements of $G$ of order at most $p^i$.
\end{abstract}

\maketitle

Let $G$ be a finite $p$-group.
For every $i\ge 0$, denote by $\Omega_{\{i\}}(G)$ the set of elements of $G$
of order at most $p^i$, and let $\Omega_i(G)$ be the subgroup it generates.
On some occasions, $\Omega_i(G)$ coincides with $\Omega_{\{i\}}(G)$, in other
words, the exponent of $\Omega_i(G)$ is at most $p^i$.
This happens, of course, for abelian groups, and also for powerful $p$-groups
if $p$ is odd, as shown by L. Wilson in \cite{Wi2}.
In his Ph.D.\ Thesis \cite{Wi1}, Wilson also obtained that
$\exp\Omega_i(G)\le 2^{i+1}$ for $p=2$, which is best possible.
On the other hand, H\'ethelyi and L\'evai \cite{HL} have proved that
$|G:G^{p^i}|=|\Omega_{\{i\}}(G)|$ for all powerful $p$-groups.
These properties of powerful $p$-groups have been generalized by
Gonz\'alez-S\'anchez and Jaikin-Zapirain \cite{GJ} to normal subgroups
lying inside $G^2$.

The aim of this note is to provide a short direct proof of the results of
Wilson and H\'ethelyi-L\'evai.
The proof is completely elementary and only uses basic properties of powerful
$p$-groups (as can be found in Chapter 2 of \cite{DDMS} or Chapter 11 of \cite{Kh}),
and the following version of Hall's collection formula: if $x$ and $y$
are elements of a group $G$, $p$ is a prime and $n\in\N$, then
\[
(xy)^{p^n} \equiv x^{p^n}y^{p^n}
\pmod{\gamma_2(H)^{p^n}\gamma_p(H)^{p^{n-1}} \ldots \gamma_{p^n}(H)}
\]
and
\[
[x,y]^{p^n} \equiv [x^{p^n},y]
\pmod{\gamma_2(K)^{p^n}\gamma_p(K)^{p^{n-1}} \ldots \gamma_{p^n}(K)},
\]
where $H=\langle x,y \rangle$ and $K=\langle x,[x,y] \rangle$.

It is convenient to make the following convention in order to write our results more
easily: if $G$ is a $p$-group and $x\in G$, let us define the meaning of the inequality
$o(x)\le p^i$ with $i<0$ (which is actually impossible to hold) to be that $x=1$.
Define accordingly $\Omega_i(G)=\Omega_{\{i\}}(G)=1$ for all $i<0$, which is coherent with
the definition given for $i\ge 0$.

\begin{thm}
\label{omega}
Let $G$ be a powerful $p$-group.
Then, for every $i\ge 0$:
\begin{enumerate}
\item
If $x,y\in G$ and $o(y)\le p^i$, then $o([x,y])\le p^i$.
\item
If $x,y\in G$ are such that $o(x)\le p^{i+1}$ and $o(y)\le p^i$, then
$o([x^{p^j},y^{p^k}])$\break
$\le p^{i-j-k}$ for all $j,k\ge 0$.
\item
If $p$ is odd, then $\exp\Omega_i(G)\le p^i$.
\item
If $p=2$, then $\exp\Omega_i(T)\le 2^i$ for any subgroup $T$ of $G$
which is cyclic over $G^2$.
In particular, $\exp\Omega_i(G^2)\le 2^i$.
\end{enumerate}
\end{thm}

\begin{proof}
We apply induction on the order of $G$ globally to all the assertions of the
theorem.
Thus we prove that (i) through (iv) hold for a group $G$ under the assumption
that they hold for all powerful groups of smaller order.
Of course, we may suppose that $G$ is non-cyclic.

(i)
The commutator $[x,y]=(y^{-1})^xy$ is a product of two elements of order
$\le p^i$ in the subgroup $T=\langle y,G'\rangle$.
If $p$ is odd, then $T$ is a proper powerful subgroup of $G$, by Lemma 11.7
of \cite{Kh}.
If $p=2$, then $H=\langle y,G^2\rangle$ is also powerful and proper in $G$,
and $T$ is a subgroup of $H$ that is cyclic over $H^2$.
Since (iii) and (iv) hold for powerful groups of order smaller than that
of $G$, the result follows.

(ii)
In this case, we also apply reverse induction on $i$.
If we write $\exp G=p^e$, then the result is clear for $i\ge e$:
just note that $[x^{p^j},y^{p^k}]=g^{p^{j+k}}$ for some $g\in G$,
since $G$ is powerful.
Suppose now $i<e$ and put $T=\langle x,[x,y^{p^k}]\rangle$.
By Hall's collection formula,
\begin{equation}
\label{Hall}
[x^{p^j},y^{p^k}]
\equiv
[x,y^{p^k}]^{p^j}
\pmod{\gamma_2(T)^{p^j}\gamma_p(T)^{p^{j-1}}\ldots\gamma_{p^j}(T)}.
\end{equation}
Since $o([x,y^{p^k}])\le o(y^{p^k})\le p^{i-k}$, it follows that
$[x,y^{p^k}]^{p^j}\in\Omega_{i-j-k}(T)$.
Note that, arguing as in the last paragraph, the induction on the group
order yields that $\exp \Omega_n(T)\le p^n$ for all $n$.
(In particular $\exp T\le p^{i+1}$.)
Thus if we prove that all the subgroups appearing in the modulus of
congruence (\ref{Hall}) lie in $\Omega_{i-j-k}(T)$, then we can conclude
that the order of $[x^{p^j},y^{p^k}]$ is at most $p^{i-j-k}$, as desired.

First, since $\gamma_2(T)=\langle [x,y^{p^k},x]\rangle^T\le\Omega_{i-k}(T)$,
we get $\gamma_2(T)^{p^j}\le\Omega_{i-j-k}(T)$.
Let us now prove by induction on $r$ that $\gamma_{r+2}(T)\le\Omega_{i-k-r}(T)$
for all $r\ge 0$.
We have just seen this for $r=0$.
In the general case, it suffices to show that if $a\in\gamma_{r+1}(T)$ and
$b\in T$ then $o([a,b])\le p^{i-k-r}$.
Since $\gamma_{r+1}(T)\le\Omega_{i-k-r+1}(T)$, we know that
$o(a)\le p^{i-k-r+1}$.
On the other hand,
$a\in\gamma_{r+1}(T)\le [G^{p^k},G,\overset{r+1}{\ldots},G]$
and $G$ being powerful imply that $a=g^{p^{k+r+1}}$ for some $g\in G$,
and then $o(g)\le p^{i+2}$.
But $b\in T$ has order $\le p^{i+1}$, so the reverse induction
on $i$ we are applying yields that
$o([a,b])=o([g^{p^{k+r+1}},b])\le p^{i-k-r}$.

Now if either $p$ is odd and $r\ge 1$ or if $p=2$ and $r\ge 2$, we have
$\gamma_{p^r}(T)\le\gamma_{r+2}(T)$, and consequently
$\gamma_{p^r}(T)^{p^{j-r}}\le\Omega_{i-j-k}(T)$.
Thus we only have to check that
$\gamma_2(T)^{2^{j-1}}\le\Omega_{i-j-k}(T)$ for $p=2$.
This follows by observing that $[x,y^{2^k},x]=[g^{2^{k+2}},x]$ for some
$g$ of order $\le 2^{i+2}$ and hence, by the reverse induction on $i$,
this commutator has order at most $2^{i-k-1}$.

(iii)
Let $x,y,z\in G$ be elements of order $p$.
Then $[x,y]=g^p$ for some $g$ of order at most $p^2$, and it follows from
(ii) that $[x,y,z]=1$.
Hence the nilpotency class of $\Omega_1(G)$ is at most $2$.
Since $p$ is odd, $\Omega_1(G)$ is then regular and $\exp\Omega_1(G)\le p$.
For the general case, by the induction on the group order, we get
$\exp\Omega_{i-1}(G/\Omega_1(G))\le p^{i-1}$.
This, together with $\exp\Omega_1(G)\le p$, yields that
$\exp\Omega_i(G)\le p^i$.

(iv)
As in (iii), it is enough to prove that $\exp\Omega_1(T)\le 2$.
(We need $T\trianglelefteq G$ in order to apply the induction hypothesis to
$G/\Omega_1(T)$, but this is not a problem, since we may assume $G^2\le T$.)
For this purpose, we are going to see that $\Omega_1(T)$ is abelian,
that is, that any two elements $a,b\in T$ of order $2$ commute.
If either $a$ or $b$ lie in $G^2$, this is a consequence of (ii), so we
suppose that $a,b\in\Omega_1(T)\setminus G^2$.
Since $T$ is cyclic over $G^2$, it follows that $b=av$ for some $v\in G^2$,
and we have to prove that $a$ and $v$ commute.
Now $v^2=(ab)^2=[b,a]=[v,a]\in G^8$, so there exists $u\in G^4$ such that
$u^2=v^2$.
We also deduce that $v^2$ has order at most $2$, whence $o(u)=o(v)\le 4$.
It is then a consequence of (ii) that $u$ and $v$ commute.
Put $w=v^{-1}u\in G^2$.
Since $w^2=1$ it follows again from (ii) that $w$ commutes with $a$ and,
of course, it also commutes with $v$.
Hence $(au)^2=(avw)^2=(av)^2=1$ and, since $a,au\in\langle a,G^4\rangle$,
we may apply the induction on the group order to deduce that $a$ and $au$
commute.
Since $au=avw$ we conclude that $a$ commutes with $v$, as desired.
\end{proof}

\begin{cor}
Let $G$ be a powerful $2$-group.
Then:
\begin{enumerate}
\item
$\exp [\Omega_i(G),G]\le 2^i$.
\item
$\exp\Omega_i(G)\le 2^{i+1}$.
\end{enumerate}
\end{cor}

\begin{proof}
(i)
By part (i) of Theorem \ref{omega}, $[\Omega_i(G),G ]$ can be generated
by elements of order at most $2^i$.
Since $G$ is powerful, it follows that $[\Omega_i(G),G]\le\Omega_i(G^2)$,
and then we only have to apply part (iv) of Theorem \ref{omega}.

(ii)
As observed above, we have $[\Omega_i(G),G]\le\Omega_i(G^2)$, and
consequently the factor group $\Omega_i(G)/\Omega_i(G^2)$ is abelian.
But this group can be generated by elements of order at most $2$, hence
its exponent is also at most $2$.
Since $\exp\Omega_i(G^2)\le 2^i$, we conclude that
$\exp\Omega_i(G)\le 2^{i+1}$.
\end{proof}

Some other properties of omega subgroups of powerful $p$-groups can be
immediately deduced from Theorem \ref{omega}, and their proof is left
to the reader.
For example, that for odd $p$ the subgroup $\Omega_1(G^p)$ is powerful
and that $\exp \gamma_{r+2}(\Omega_i(G))\le p^{i-r}$ for all $r\ge 0$.
In particular, the nilpotency class of $\Omega_i(G)$ is at most $i+1$
for $p>2$.

Finally, we prove the result of H\'ethelyi and L\'evai.
We need the following lemma, which is an immediate consequence of Hall's
collection formula.

\begin{lem}
Let $G$ be a powerful $p$-group of exponent $p^e$.
Then for every $0\le i\le e-1$, and every $x\in G$, $y\in G^{p^{e-i-1}}$,
we have $(xy)^{p^i}=x^{p^i}y^{p^i}$.
\end{lem}

\begin{thm}
Let $G$ be a powerful $p$-group.
Then $|G:G^{p^i}|=|\Omega_{\{i\}}(G)|$ for all $i\ge 0$.
\end{thm}

\begin{proof}
We argue by induction on the group order.
Let $\exp G=p^e$.
Of course, we only need to consider the case when $i\le e-1$.
On the other hand, by the previous lemma, the map $x\mapsto x^{p^{e-1}}$
is a homomorphism from $G$ onto $G^{p^{e-1}}$, and by the first isomorphism
theorem, $|\Omega_{\{e-1\}}(G)|=|G:G^{p^{e-1}}|$.

So we assume in the remainder that $i\le e-2$.
Set $\overline G=G/G^{p^{e-1}}$.
By the induction hypothesis,
$|G:G^{p^i}|=|\overline G:\overline G^{p^i}|=
|\Omega_{\{i\}}(\overline G)|$.
Hence if $X=\{x\in G\mid x^{p^i}\in G^{p^{e-1}}\}$ then
$|G:G^{p^i}|=|X|/|G^{p^{e-1}}|$.
Thus it suffices to prove that $|X|=|\Omega_{\{i\}}(G)||G^{p^{e-1}}|$.

Note that $x\in X$ if and only if there exists $y\in G^{p^{e-i-1}}$
such that $x^{p^i}=y^{p^i}$ and, by the previous lemma, this is
equivalent to $xy^{-1}\in\Omega_{\{i\}}(G)$.
Thus $X=\Omega_{\{i\}}(G)G^{p^{e-i-1}}$ is the union of all the cosets
of elements of $\Omega_{\{i\}}(G)$ with respect to the subgroup
$G^{p^{e-i-1}}$, and in order to get the cardinality of $X$ we have
to find out how many different cosets of that kind there are.
Let $gG^{p^{e-i-1}}$ be a coset with $g\in\Omega_{\{i\}}(G)$.
Again by the lemma, the elements of order at most $p^i$ in that coset
are exactly those in
$g(\Omega_{\{i\}}(G)\cap G^{p^{e-i-1}})=g\Omega_{\{i\}}(G^{p^{e-i-1}})$,
so their number is $|\Omega_{\{i\}}(G^{p^{e-i-1}})|$, which is
independent of the choice of $g$.
In other words, the elements of $\Omega_{\{i\}}(G)$ are evenly distributed in
the different cosets with respect to the subgroup $G^{p^{e-i-1}}$.
Hence the number of cosets we wanted is
$|\Omega_{\{i\}}(G)|/|\Omega_{\{i\}}(G^{p^{e-i-1}})|$ and
\begin{equation}
\label{cardinality X}
|X|
=
\frac{|\Omega_{\{i\}}(G)|}{|\Omega_{\{i\}}(G^{p^{e-i-1}})|}
\, |G^{p^{e-i-1}}|.
\end{equation}
Finally, since $i\le e-2$, it follows that $G^{p^{e-i-1}}$ is a proper
powerful subgroup of $G$.
By the induction hypothesis, we have
$|\Omega_{\{i\}}(G^{p^{e-i-1}})|=|G^{p^{e-i-1}}:G^{p^{e-1}}|$,
and we get from (\ref{cardinality X}) that
$|X|=|\Omega_{\{i\}}(G)||G^{p^{e-1}}|$, as desired.
\end{proof}


\begin{thebibliography}{9}

\bibitem{DDMS}
J.D. Dixon, M.P.F. du Sautoy, A. Mann and D. Segal,
\textit{Analytic pro-$p$ groups\/}, 2nd edition,
Cambridge University Press, 1999.

\bibitem{GJ}
J. Gonz\'alez-S\'anchez and A. Jaikin-Zapirain,
On the structure of normal subgroups of potent $p$-groups,
{\it J. Algebra\/} {\bf 276} (2004), 193--209.

\bibitem{HL}
L. H\'ethelyi and L. L\'evai,
On elements of order $p$ in powerful $p$-groups,
\textit{J. Algebra\/} \textbf{270} (2003), 1--6.

\bibitem{Kh}
E.I. Khukhro,
\textit{$p$-Automorphisms of finite $p$-groups\/},
Cambridge University Press, 1998.

\bibitem{Wi1}
L. Wilson,
\textit{Powerful groups of prime power order\/},
Ph.D.\ Thesis, University of Chicago, 2002.

\bibitem{Wi2}
L. Wilson,
On the power structure of powerful $p$-groups,
\textit{J. Group Theory\/} \textbf{5} (2002), 129--144.

\end{thebibliography}
\end{document}